\documentclass{amsart}
\usepackage{amsfonts,amssymb,amsmath,amsthm}
\usepackage{url}
\usepackage{enumerate}
\usepackage{mathrsfs}

\newtheorem{thrm}{Theorem}[section]
\newtheorem{lem}[thrm]{Lemma}
\newtheorem{prop}[thrm]{Proposition}

\theoremstyle{definition}

\newtheorem{remark}[thrm]{Remark}
\numberwithin{equation}{section}

\author{L. Deleaval}
\address{Institut Math\'ematiques de Jussieu\\ Universit\'e Paris VI\\
France}
\email{deleaval@math.jussieu.fr}
\author{N. Demni}
\address{Institut de Recherche en Math\'ematiques de Rennes\\ Universit\'e Rennes 1\\
France}
\email{nizar.demni@univ-rennes1.fr}

\keywords{Conditional independence, normal distribution, Bessel-type hypergeometric functions, product formulas}
\subjclass{Primary 42A61, Secondary 33C70}
\begin{document}

\title[Product Formula]{A probabilistic proof of product formulas for spherical Bessel functions and their matrix analogues}

\begin{abstract}
We write, for geometric index values, a probabilistic proof of the product formula for spherical Bessel functions. Our proof has the merit to carry over without any further effort to Bessel-type hypergeometric functions of one matrix argument. Moreover, the representative probability distribution involved in the matrix setting is shown to be closely related to matrix-variate normal distributions and to the symmetrization of upper-left corners of Haar-distributed orthogonal matrices. Once we did, we use the latter relation to perform a detailed analysis of  this probability distribution. In case it is absolutely continuous with respect to Lebesgue measure on the space of real symmetric matrices, the product formula for Bessel-type hypergeometric functions of two matrix arguments is obtained from Weyl integration formula. 
\end{abstract}
\maketitle

\section{Reminder and motivation}
The spherical Bessel function $j_{\nu}$ of index $\nu$ is defined for all complex $z$ and all $\nu > -1$ by (\cite{Watson}) 
\begin{equation*}
j_{\nu}(z) = \sum_{l=0}^{+\infty} \frac{(-1)^l}{(\nu+1)_l l!} \left(\frac{z}{2}\right)^{2l},
\end{equation*}
where $(\nu+1)_{l}: = \Gamma(\nu+l+1)/\Gamma(\nu+1)$ denotes the usual Pochhammer symbol. It provides a basic example of one-variable special function satisfying a product formula that opened the way to a rich harmonic analysis. More precisely, for $\nu \geq -1/2$ and nonnegative real numbers $x,y,z$, it is well known that
\begin{equation}\label{PF}
j_{\nu}(xy)j_{\nu}(zy) = \int_{\mathbb R_+} j_{\nu}(\xi y) \tau_{x,z}^{\nu}(d\xi),  
\end{equation}
where $\tau_{x,z}^{\nu}$ is a compactly-supported probability distribution. Recall that for $\nu > -1/2$, \eqref{PF} is a trivial consequence of the addition Theorem for Bessel functions (see for instance Chapter XI in \cite{Watson}) while it obviously holds for $\nu = -1/2$ since $j_{-1/2}(z) = \cos(z)$. Nevertheless, for integer $p \geq 1$ and for the so-called geometrical index values $\nu = (p/2) -1$, \eqref{PF} may be derived from the following Poisson-type integral representation 
\begin{equation}
j_{(p/2)-1}(|v|) = \int_{S^{p-1}} e^{i\langle v, s \rangle}\sigma_1(ds), \quad  v\in \mathbb R^p,  
\end{equation}
where $\sigma_1$ is the uniform distribution on the unit sphere $S^{p-1}$ and $\mathopen\langle\cdot,\cdot\mathclose\rangle$, $\mathopen|\cdot\mathclose|$ are  respectively the Euclidean inner product and the associated Euclidean norm in $\mathbb{R}^p$. Indeed, if we set $|v| = y$, then 
\begin{align*}
j_{(p/2)-1}(x|v|) j_{(p/2)-1}(z|v|) = \int_{\mathbb{R}^p} e^{i \langle v, s \rangle} (\sigma_x \star \sigma_z)(ds),
\end{align*} 
where $\sigma_x,\sigma_z$ are the uniform distributions on spheres of radii $x,z$ respectively. But according to \cite{Rag} Corollary 5.2 p.1149, the probability distribution $\sigma_x \star \sigma_z$ is absolutely continuous with respect to the Lebesgue measure in $\mathbb{R}^p$ and due to its rotational invariance it has a radial density. The use of spherical coordinates yields then \eqref{PF}. Avoiding techniques from differential geometry like the ones used to prove the absolute continuity of $\sigma_x \star \sigma_z$, we write a probabilistic proof of \eqref{PF} for geometric index values and supply a probabilistic interpretation of $\tau_{x,z}^{(p/2)-1}$. Our starting point is the elementary fact that the conditional distribution of a standard normal vector $N$ in $\mathbb{R}^p$ given its radius $|N|$ is the uniform distribution on the sphere of radius $|N|$. The product of two spherical Bessel functions turns towards the conditional independence of two independent standard normal vectors $N_1,N_2$ relative to the $\sigma$-field generated by their radii $|N_1|, |N_2|$ (\cite{Rev}). The representative probability distribution $\tau_{x,z}^{(p/2)-1}$ is then seen to be the conditional distribution of the radial part $|N_1+N_2|$ given $(|N_1| = x, |N_2| = z)$. In fact, $N_1+N_2$ is again distributed as a standard Gaussian vector (up to a constant) and its angular part is independent from both radii $|N_1|$ and $|N_2|$. The reader will easily realize from the ingredients needed in the proof that choosing any multivariate stable distribution in $\mathbb{R}^p$ whose density is a radial function does not alter our proof.  But the Fourier transform of a radial function is again radial therefore the choice restricts uniquely to isotropic or rotationally invariant stable distributions (whose L\'evy exponents are given up to a constant by $v \mapsto |v|^{\alpha}, \alpha \in (0,2]$, \cite{Sato} p.86). 

Our proof has also the merit to carry over after mild modifications to some matrix analogues of spherical Bessel functions. Those we consider here are known as Bessel-type hypergeometric functions of one and two $m \times m$ real symmetric matrix arguments. The product formulas we obtain are valid for geometrical index values and are those derived in \cite{Ros} using hypergroup theory, in the particular case of the real division algebra. This is by no means a loss of generality since product formulas over the division algebra $\mathbb{C}$ may be easily derived along the same lines. For functions of  one matrix argument, the proof is identical to that written for $j_{(p/2)-1}$. Besides, the representative probability distribution is seen to be the conditional distribution of the radial part of the sum of two independent $p \times m \, (p \geq m)$ standard matrix-variate normal distributions given the radial part of each. We shall prove that this conditional distribution is closely related to the distribution of the $m \times m$ upper-left corner of an orthogonal matrix of size $p$, whence its absolute continuity (with respect to Lebesgue measure) is deduced for $p \geq m+1$. For these values of $p$, one easily derives the product formula for functions of two arguments using Weyl integration formula for the space of real symmetric matrices. As a matter of fact, the corresponding representative probability distribution has an analogous description in terms of singular values rather than matrices. Besides, when $p \geq 2m$, a result due to B. Collins provides a detailed description of the distribution of the upper-left corner of an orthogonal matrix, agreeing with the variable change formula given in Lemma 3.7 p.495 in \cite{Herz} and reproved in Corollary 3.3 p. 762 in \cite{Ros}. Note finally that since Bessel-type hypergeometric functions of two matrix arguments we consider here are instances of generalized Bessel functions associated with $B$-type root systems, then our approach resembles the one carried for proving Theorem 5.16 (ii) in \cite{BBO}.  

The paper is organized as follows. In the next section, we consider spherical Bessel functions $j_{(p/2)-1}$ and prove \eqref{PF} for geometric index values. In section 3, we extend our proof to Bessel-type hypergeometric functions of one real symmetric matrix argument. In the last section, we perform a detailed analysis of the representative probability distribution: it is absolutely continuous for $p \geq m+1$ and its density enjoys a certain averaged bi-invariance property with respect to the orthogonal group. The product formula for functions of two real symmetric matrix arguments follows from Weyl integration formula.  

\section{Product formula for spherical Bessel functions}
All random variables occuring below are defined on some probability space $(\Omega, \mathscr{F}, \mathbb{P})$ and we denote $\mathbb{E}$ the corresponding expectation. Furthermore, for the $\sigma$-field $\sigma(X)$ generated by a random variable $X$, we write 
\begin{equation*}
\mathbb{E}[\mathopen \cdot \mathclose|X] \quad \textrm{for} \quad \mathbb{E}[\mathopen \cdot \mathclose|\sigma(X)],
\end{equation*}
and we recall that all equalities involving conditional expectations hold $\mathbb{P}$-almost surely. Let $N$ be a standard normal vector\footnote{Its coordinates are independent centered normal distributions with unit variance.} in $\mathbb{R}^p$ and let $N = R\Theta$ be its polar decomposition ($R > 0$ and  $\Theta \in S^{p-1}$). Then, $R$ and $\Theta$ are independent and $\Theta$ is uniformly distributed on $S^{p-1}$. It follows that for any $v \in \mathbb{R}^p$
\begin{equation*}
\mathbb{E}\left[e^{i\langle v, N\rangle}|R\right] = \int_{S^{p-1}}e^{i\langle v, Rs\rangle}\sigma_1(ds) = j_{(p/2)-1}(|v|R). 
\end{equation*}
In fact, if $X,Y$ are independent random variables valued in some measurable spaces and if $\mathcal{D}_Y$ stands for the distribution of $Y$, then 
\begin{equation*}
\mathbb{E}[f(X,Y)| X] = \int f(X,y) \mathcal{D}_Y(dy)
\end{equation*}
for any bounded Borel function $f$ (see \cite{Rev} p.108 Exercice 4.27). 

Now, let $N_1, N_2$ be two independent standard normal vectors in $\mathbb{R}^p$ with polar decompositions $N_1 =  R_1\Theta_1, N_2 = R_2\Theta_2$ respectively, and consider the product $\sigma$-field $\sigma(R_1,R_2)$ generated by $R_1,R_2$. Then, the independence of $N_1$ and $N_2$ implies that (\cite{Rev}) 
\begin{eqnarray*}
\mathbb{E}\left[e^{i\langle v, N_1\rangle}|R_1\right] &=& \mathbb{E}\left[e^{i\langle v, N_1\rangle}| R_1,R_2\right]\\
\mathbb{E}\left[e^{i\langle v, N_2\rangle}|R_2\right] &=& \mathbb{E}\left[e^{i\langle v, N_2\rangle}|R_1,R_2\right].
\end{eqnarray*}
Besides, $N_1,N_2$ are conditionally independent relative to $\sigma(R_1,R_2)$ (see \cite{Rev} p.109 Exercice 4.32). In fact, one has for any bounded Borel function $f: \mathbb{R}^p \to \mathbb{R}$ 
\begin{align*}
\mathbb{E}\Bigl[f(N_2)| N_1, R_1,R_2\Bigr]  = \mathbb{E}\Bigl[f(N_2)|R_2\Bigr]  =  \mathbb{E}\Bigl[f(N_2)|R_1,R_2\Bigr].
\end{align*}
Thus
\begin{equation*}
\mathbb{E}\left[e^{i\langle v, N_1\rangle}|R_1\right] \mathbb{E}\left[e^{i\langle v, N_2\rangle}|R_2\right] = \mathbb{E}\left[e^{i\langle v, N_1+N_2\rangle}|R_1,R_2\right].
\end{equation*}
Write $N_1+N_2: = R_3\Theta_3$, then $N_1+N_2$ is (up to a constant factor) a standard normal vector so that $\Theta_3$ is uniformly distributed on $S^{p-1}$ and is independent from $R_3$. 
We claim that: 

\begin{prop}\label{L1}
$\Theta_3$ is independent from $\sigma(R_1,R_2)$.
\end{prop}
\begin{proof} Let $f: S^{p-1} \to \mathbb{R}, g: \mathbb{R}_+ \times \mathbb{R}_+ \to \mathbb{R}$ be bounded Borel functions, then the independence of $N_1,N_2$ yields 
\begin{multline*}
\mathbb{E}\left[f(\Theta_3)g(R_1,R_2)\right] = \mathbb{E}\left[f\left(\frac{N_1+N_2}{|N_1+N_2|}\right)g(|N_1|,|N_2|)\right] \\
= \int_0^{\infty}\int_0^{\infty}F(r_1,r_2)dr_1dr_2\int_{S^{p-1}\times S^{p-1}}f\left(\frac{r_1\theta_1+r_2\theta_2}{|r_1\theta_1+r_2\theta_2|}\right)\sigma_1(d\theta_1)\sigma_1(d\theta_2),
\end{multline*}
where
\begin{equation*}
F(r_1,r_2) := (r_1r_2)^{p-1}e^{-(r_1^2+r_2^2)/2}g(r_1,r_2).
\end{equation*}
Let $\nu_{r_1,r_2}(d\theta)$ be the pushforward of $\sigma_1 \otimes \sigma_1$ under the map 
 \begin{equation*}
 (\theta_1,\theta_2) \mapsto \frac{r_1\theta_1+r_2\theta_2}{|r_1\theta_1+r_2\theta_2|},
 \end{equation*}
then
\begin{equation*}
\int_{S^{p-1}\times S^{p-1}}f\left(\frac{r_1\theta_1+r_2\theta_2}{|r_1\theta_1+r_2\theta_2|}\right)\sigma_1(d\theta_1)\sigma_1(d\theta_2) = \int_{S^{p-1}}f\left(\theta\right)\nu_{r_1,r_2}(d\theta).
 \end{equation*}
But $\nu_{r_1,r_2}$ is obviously invariant under the action of $O(p)$, therefore $\nu_{r_1,r_2} = \sigma_1$ since $\sigma_1$ is the unique distribution on $S^{p-1}$ enjoying the rotational invariance property. 
\end{proof}

We also need the following lemma:
\begin{lem}\label{L2}
Let $V,X,Y$ be random variables such that $Y$ and $(X,V)$ are independent. Then, for any bounded Borel function $f$ 
\begin{equation*}
\mathbb{E}[f(X,Y)| V] = \int \mathbb{E}[f(X,y)|V] \mathcal{D}_Y(dy).
\end{equation*}
\end{lem}
\begin{proof} 
This fact is easily proved for bounded functions $f(x,y) = g(x)h(y)$ and then extended to bounded Borel functions using the monotone class Theorem (\cite{Rev1} p.5). 
\end{proof}

Combining the proposition and the lemma, one gets 
 \begin{align*}
 \mathbb{E}\left[e^{i\langle v, N_1+N_2\rangle}|R_1,R_2\right] = \int_{S^{p-1}}\mathbb{E}\left[e^{i\langle v, R_3 s\rangle}|R_1,R_2\right] \sigma_1(ds).
 \end{align*}
Finally, let $\mu_{R_3|(R_1,R_2)}$ be a regular version of the conditional distribution  of $R_3$ given $(R_1,R_2)$, then Fubini Theorem entails  
\begin{equation*}
j_{(p/2)-1}(|v|R_1)  j_{(p/2)-1}(|v|R_2)  = \int_{\mathbb{R}_+}  j_{(p/2)-1}(|v|\xi) \mu_{R_3|(R_1,R_2)}(d\xi).
\end{equation*} 
Thus, \eqref{PF} is proved and $\tau_{x,z}^{(p/2)-1}$ fits $\mu_{R_3|(R_1,R_2)}$ on the event $\{R_1=x,R_2=z\}$ as explained in the following remark.

\begin{remark}
Let $\Phi$ be the angle between $\Theta_1, \Theta_2$: $\cos\Phi = \langle \Theta_1,\Theta_2\rangle$. Then 
 \begin{equation*}
 R_3 = \sqrt{R_1^2 + R_2^2 + 2R_1R_2\cos\Phi}.
 \end{equation*}
But the independence of $\Theta_1, \Theta_2$ entails for any real $w$
\begin{align*}
\mathbb{E}[e^{iw\cos \Phi}] &= \int_{S^{p-1}}\int_{S^{p-1}} e^{iw \langle s,t\rangle} \sigma_1(ds)\sigma_1(dt)
\\& = \int_{S^{p-1}} j_{(p/2)-1}(w|t|)\sigma_1(dt) 
\\& = j_{(p/2)-1}(w) = \frac{\Gamma(p/2)}{\Gamma(1/2)\Gamma((p-1)/2)} \int_{-1}^{1} e^{iw\xi}(1-\xi^2)^{(p-3)/2} d\xi
\end{align*}
where we used Lemma 5.4.4 p.195  in \cite{DX}. Performing the variable change 
\begin{equation*}
u = \sqrt{x^2+z^2 + 2xz\xi}, \quad \xi \in [-1,1], 
\end{equation*}
one recovers the density of $\tau_{x,z}^{(p/2)-1}$ derived in Proposition A.5. p. 1153 in \cite{Rag}. 
\end{remark}

\section{Product formula for Bessel-type hypergeometric functions of one real symmetric matrix argument}
In this section, we consider matrix-variate normal distributions rather than vectors. Doing so leads to a product formula for Bessel-type hypergeometric functions of one real symmetric matrix argument (see below). To this end, we recall from \cite{Chikuse} Ch.I. the following needed facts. Let $p\geq m \geq 1$ and let $N$ be a real matrix-variate $p \times m$ standard normal distribution, that is a $p \times m$ matrix whose entries are independent centered normal distributions with unit variance. Then $N$ admits almost surely a unique polar decomposition $N = Z(N^TN)^{1/2} := ZH$. Moreover, $Z$ and $H$ are independent, $H$ is almost surely invertible and $Z$ is uniformly distributed on the real Stiefel manifold 
\begin{equation*}
\Sigma_{p,m}:=  \{A \in M_{p,m}(\mathbb{R}), A^TA = {\bf I}_m\},
\end{equation*}
where $M_{p,m}(\mathbb{R})$ is the space of $p \times m$ real matrices. Let $O(p)$ be the orthogonal group, then $\Sigma_{p,m}$ is a homogeneous space $\Sigma_{p,m} \approx O(p)/O(p-m)$. It thereby admits a unique $O(p)$-invariant distribution we shall denote $\sigma_{p,m}$. More precisely, $\sigma_{p,m}$ is the pushforward of the Haar distribution on $O(p)$ under the map 
\begin{equation*}
O \mapsto Oe_{p,m}, \quad e_{p,m}:= {\it I}_m \oplus 0_{p-m,m}.
\end{equation*} 
Hence, for any $C \in M_{p,m}(\mathbb{R})$  
\begin{align*}
\mathbb{E}\left[e^{i\textrm{tr}(C^T N)}| \, H \right] &= \int_{\Sigma_{p,m}} e^{i\textrm{tr}(C^T sH)}\sigma_{p,m}(ds) = \int_{\Sigma_{p,m}} e^{i\textrm{tr}(HC^T s)}\sigma_{p,m}(ds).
\end{align*}
Now, let $N_1,N_2$ be two independent $p\times m$ matrix-variate standard normal distributions with corresponding polar decomposition $N_1=Z_1H_1,N_2=Z_2H_2$. Then, by considering the product $\sigma$-field $\sigma(H_1,H_2)$ generated by $H_1, H_2$ we easily derive 
\begin{equation} \label{bes}
\mathbb{E}\left[e^{2i\textrm{tr}(C^T N_1)}|H_1\right] \mathbb{E}\left[e^{2i\textrm{tr}(C^T N_2)}|H_2\right] = \mathbb{E}\left[e^{2i\textrm{tr}(C^T (N_1+N_2))}|H_1,H_2\right].
\end{equation}
Since $N_1+N_2$ is up to a constant factor a $p\times m$ matrix-variate standard normal distribution, then it admits almost surely a polar decomposition $N_1+N_2 = Z_3H_3$, where $Z_3$ is uniformly distributed on $\Sigma_{p,m}$ and is independent from $H_3$. Similarly to the case $m=1$, one proves that $Z_3$ is independent from $\sigma(H_1,H_2)$ (analogue of proposition \ref{L1}) using the following variable change formula (\cite{FK}, Prop. XVI.2.1. p.351): let $dA$ be the Lebesgue measure on $M_{p,m}(\mathbb{R})$, let $S_m^+(\mathbb{R})$ be the set of real positive definite matrices with Lebesgue measure $dr$ and $\gamma=(p/2)-1-[m(m-1)]/2$. Then 
\begin{equation*}
\int_{M_{p,m}(\mathbb{R})} f(A)dA = \int_{\Sigma_{p,m}} \int_{S_m^{+}(\mathbb{R})} f(s\sqrt{r}) [\textrm{det}(r)]^{\gamma} \sigma_{p,m}(ds)dr.
\end{equation*}
Accordingly and with the help of lemma \ref{L2}, one gets
\begin{align*}
 \mathbb{E}\left[e^{2i\textrm{tr}(C^T Z_3H_3)}|H_1,H_2\right] = \int_{\Sigma_{p,m}}\mathbb{E}\left[e^{2i\textrm{tr}(C^T sH_3)}|H_1,H_2\right] \sigma_{p,m}(ds),
\end{align*}
and if $\mu_{H_3|(H_1,H_2)}$ is the conditional distribution of $H_3$ given $(H_1,H_2)$, then Fubini Theorem entails
\begin{multline*}
  \mathbb{E}\left[e^{2i\textrm{tr}(C^T Z_3H_3)}|H_1,H_2\right]  =  \int_{S_m^+(\mathbb{R})}  \left[\int_{\Sigma_{p,m}} e^{2i\textrm{tr}(C^T s\xi)}\sigma_{p,m}(ds)\right]\mu_{H_3|(H_1,H_2)}(d\xi).
\end{multline*}
Using \cite{Herz}, (3.5) p.493, one sees that    
\begin{equation*}
\mathbb{E}\left[e^{2i\textrm{tr}(C^T N)}| H \right] = \int_{\Sigma_{p,m}} e^{2i\textrm{tr}(HC^T s)}\sigma_{d,m}(ds) = {}_0F_1\left(\frac{p}{2}; -(HC^TCH)\right)
\end{equation*}
where ${}_0F_1$ is the Bessel-type hypergeometric function of one real symmetric argument and of geometrical index value $(p/2)$ (it reduces when $m=1$ to $j_{(p/2)-1}$, \cite{Muir}).
Finally, \eqref{bes} yields the product formula
\begin{multline*}
{}_0F_1\left(\frac{p}{2}; -H_1C^TCH_1\right){}_0F_1\left(\frac{p}{2}; -H_2C^TCH_2\right) \\ = \int_{S_m^+(\mathbb{R})}{}_0F_1\left(\frac{p}{2}; -\xi C^TC\xi\right) \mu_{H_3|(H_1,H_2)}(d\xi).
\end{multline*}

Now, we proceed to 
\section{Absolute continuity of $\mu_{H_3|(H_1,H_2)}$ and Product formula for Bessel-type hypergeometric functions of two matrix arguments}
\subsection{Absolute continuity of $\mu_{H_3|(H_1,H_2)}$} 
In contrast to the case $m=1$, the absolute-continuity of $\mu_{H_3|(H_1,H_2)}$ is not obvious and needs a careful analysis we perform below:  
\begin{prop}
For any $p \geq m+1$, $\mu_{H_3|(H_2,H_1)}$ is absolutely continuous with respect to the Lebsegue measure on $S_m(\mathbb{R})$ and its density, say $f_{(H_1,H_2)}(A)$, satisfies: 
\begin{equation}\label{P}
\int_{O(m) \times O(m)} f_{(O_1H_1O_1^T,O_2H_2O_2^T)}(O_3^TAO_3)dO\otimes dO = \int_{O(m) \times O(m)} f_{(O_1H_1O_1^T, O_2H_2O_2^T)}(A)dO \otimes dO
\end{equation}
almost surely for any $O_3 \in O(m)$, where $dO$ is the Haar distribution on $O(m)$. For $p=m$, it is singular. 
\end{prop}

\begin{proof} Since
\begin{equation*}
(H_3)^2 = (H_1)^2 + (H_2)^2 + H_1Z_1^TZ_2H_2 + H_2Z_2^TZ_1H_1
\end{equation*}
then $\mu_{H_3|(H_2,H_1)}$ is the pushforward of $\sigma_{p,m} \otimes \sigma_{p,m}$ under the map 
\begin{equation*}
(Z_1,Z_2) \mapsto \sqrt{(H_1)^2 + (H_2)^2 + H_1Z_1^TZ_2H_2 + H_2Z_2^TZ_1H_1}
\end{equation*}
for fixed $H_1,H_2$, where for a positive semi-definite matrix $A$, $\sqrt{A}$ is its square root. But from the very definition of $\sigma_{p,m}$, $\mu_{H_3|(H_1,H_2)}$ is the pushforward of the Haar distribution $dO \otimes dO$ on $O(p) \times O(p)$ under the map  
\begin{equation*}
(O_1,O_2)  \mapsto \sqrt{(H_1)^2 + (H_2)^2 + H_1e_{p,m}^TO_1^TO_2e_{p,m}H_2 + H_2e_{p,m}^TO_2^TO_1e_{p,m}H_1}
\end{equation*}
 or equivalently 
\begin{equation*}
(O_1,O_2) \mapsto \sqrt{(H_1)^2 + (H_2)^2 + H_1e_{p,m}^TO_1O_2e_{p,m}H_2 + H_2e_{p,m}^TO_2^TO_1^Te_{p,m}H_1}
\end{equation*}
since $dO$ is invariant under $O \mapsto O^T$.  Besides, the random variable $O_1O_2 \in O(p)$  is  Haar distributed since it is $O(p)$-invariant. 
As a matter of fact, $\mu_{H_3|(H_1,H_2)}$ is the pushforward of $dO$ under the map 
\begin{equation*}
O \mapsto \sqrt{(H_1)^2 + (H_2)^2+ H_1e_{p,m}^TOe_{p,m}H_2 + H_2e_{p,m}^TO^Te_{p,m}H_1}.
\end{equation*}
Now observe that for fixed $H_1,H_2$, 
\begin{equation*}
O \mapsto (H_1)^2 + (H_2)^2+ H_1e_{p,m}^TOe_{p,m}H_2 + H_2e_{p,m}^TO^Te_{p,m}H_1
\end{equation*}
is a affine map from $O(p)$ into $S_m(\mathbb{R})$, therefore is lipschitzian whose differential map is constant.  Moreover $O(p)$ and $S_m(\mathbb{R})$ are real analytic manifolds such that dim $O(p) = p(p-1)/2$, dim $S_m(\mathbb{R}) = m(m+1)/2$. As a matter of fact 
\begin{itemize}
\item If $p =m+1$, then dim $O(m+1)$ = dim $S_m(\mathbb{R})$ and Theorem 3.2.5 p.244 in \cite{Fed} implies that the pushforward of the Haar distribution on $O(p)$ under this map is absolutely continuous with respect to the Lebesgue measure on $S_m(\mathbb{R})$. 
\item If $p \geq m+2$,  then dim $O(p) >$ dim $S_m(\mathbb{R})$ and Theorem 3.2.12 p. 249 in \cite{Fed} yields the same conclusion.  
\end{itemize}
Now, for any $O_1, O_2, O_3 \in O(m)$,  $f_{(O_1H_1O_1^T,  O_2H_2O_2^T)}(O_3^TAO_3)$ is the density of the random variable (for fixed $H_1,H_2$)
\begin{multline*}
O_3O_1 (H_1)^2O_1^TO_3^T + O_3O_2(H_2)^2O_2^TO_3^T + \\ O_3O_1H_1O_1^TZ_1^TZ_2O_2H_2O_2^TO_3^T + O_3O_2H_2O_2^TZ_2^TZ_1O_1H_1O_1^TO_3^T
\end{multline*}
which can be written as
\begin{multline*}
(O_3O_1) (H_1)^2(O_1^TO_3^T) + (O_3O_2)(H_2)^2(O_2^TO_3^T) + \\ (O_3O_1)H_1(O_1^TO_3^T) (Z_1O_3^T)^T(Z_2O_3^T)(O_3O_2)H_2(O_2^TO_3^T) \\ +  (O_3O_2)H_2(O_2^TO_3^T)(Z_2O_3^T)^T (Z_1O_3^T)(O_3O_1)H_1(O_1^TO_3^T).
\end{multline*}
But since $\sigma_{p,m}$ is invariant under the right action of $O(m)$ (\cite{Chikuse} p.28) and since the Haar distribution $dO$ is $O(m)$-bi-invariant, then the $f_{(H_1, H_2)}$ satisfies \eqref{P}. 
Finally, since dim $O(m) <$ dim $S_m(\mathbb{R})$ then Theorem 3.2.5 in \cite{Fed} shows that for $p=m$, $\mu_{H_3|(H_1,H_2)}$ is singular with respect to the Lebesgue measure on $S_m(\mathbb{R})$.   
\end{proof}

\begin{remark}
Note that 
\begin{equation*}
e_{p,m}^TOe_{p,m} = \Lambda_m \oplus 0_{p-m, p-m}
\end{equation*}
where $\Lambda_m$  is the upper-left $m \times m$ corner of the orthogonal matrix $O$. According to \cite{Col}, Remark 2.1. p.118, if $p \geq 2m$ then the distribution of $\Lambda_m$ is absolutely continuous with respect to the Lebesgue measure on $M_{m,m}(\mathbb{R})$: its density is given by  
\begin{equation*}
\textrm{det}({\bf I}_m - AA^T)^{(p-2m - 1)/2}{\bf 1}_{\{||A|| < 1\}}
\end{equation*} 
where $||\mathopen \cdot \mathclose||$ is the matrix norm induced by the Euclidian norm $|\mathopen \cdot \mathclose|$. This fact should be compared with Lemma 3.7 p.495 in \cite{Herz}. 
\end{remark}

\subsection{Product formula for functions of two matrix arguments} 
Let $p \geq m+1$ so that $\mu_{H_3|(H_2,H_1)}$ is absolutely continuous with respect to Lebesgue measure on $S_m(\mathbb{R})$. Then one derives a product formula for the Bessel-type hypergeometric functions of two real symmetric matrix arguments and of geometrical index values $p/2, p \geq 1$: if $A$ is a real positive semi-definite matrix and $C \in M_{p,m}(\mathbb{R})$, then 
these functions are related to those of one real symmetric matrix argument by
\begin{equation}\label{Rel}
{}_0F_1\left(\frac{p}{2}; A; -C^TC\right) = \int_{O(m)} {}_0F_1\left(\frac{p}{2}; -O\sqrt{A}O^T(C^TC)O\sqrt{A}O^T\right) dO
\end{equation}  
where $dO$ is now the Haar distribution on $O(m)$ (Theorem 7.3.3 p. 260 in \cite{Muir}). 
Keeping the same notations used in the previous section, one has 
\begin{equation*}
{}_0F_1\left(\frac{p}{2}; A; -C^TC\right) = \int_{O(m)}\mathbb{E}\left[e^{2i\textrm{tr}(C^T N)}| H = O\sqrt{A}O^T \right] dO
\end{equation*}
which in turn implies that for any positive semi-definite matrices $A,B$ and any $C \in M_{p,m}(\mathbb{R})$  
\begin{multline*}
{}_0F_1\left(\frac{p}{2}; A; -C^TC\right){}_0F_1\left(\frac{p}{2}; B; -C^TC\right) = \\ \int_{O(m) \times O(m)} \int_{S_m^+(\mathbb{R})} {}_0F_1\left(\frac{p}{2}; -\xi C^TC\xi\right) \mu_{H_3|(O_1\sqrt{A}O_1^T,O_2\sqrt{B}O_2^T)}(d\xi)dO\otimes dO.
\end{multline*}
Recall now that $f_{(H_1,H_2)}$ denote the density of $\mu_{H_3|(H_1,H_2)}$. Then Weyl integration formula for $S_m(\mathbb{R})$ (\cite{Faraut} Theorem 10.1.1. p.232), \eqref{P} and Fubini Theorem entail 
\begin{multline*}
\int_{O(m) \times O(m)} \int_{S_m^+(\mathbb{R})} {}_0F_1\left(\frac{p}{2}; -\xi C^TC\xi\right) f_{(O_1\sqrt{A}O_1^T,O_2\sqrt{B}O_2^T)}(\xi) d\xi \otimes dO \otimes dO = c_m  \int_{O(m) \times O(m)} 
\\ \int_{O(m) \times \mathbb{R}_+^m} {}_0F_1\left(\frac{p}{2}; -ODO^T (C^TC)ODO^T\right) f_{(O_1\sqrt{A}O_1^T,O_2\sqrt{B}O_2^T)}(ODO^T)V(D)dD\otimes dO \otimes dO \otimes dO 
\\ = c_m  \int_{O(m) \times \mathbb{R}_+^m}  {}_0F_1\left(\frac{p}{2}; -ODO^T (C^TC)ODO^T\right) 
\\ \left\{\int_{O(m) \times O(m)}  f_{(O_1\sqrt{A}O_1^T,O_2\sqrt{B}O_2^T)}(D)\,dO \otimes dO\right\} V(D)dD \otimes dO 
\end{multline*}
where $D = \textrm{diag}(\lambda_1 > \lambda_2 > \ldots > \lambda_m )$ is a positive definite diagonal matrix, 
\begin{equation*}
V(D) := \prod_{1 \leq n < j \leq m} (\lambda_n - \lambda_j), \quad dD = \prod_{j=1}^m d\lambda_j,
\end{equation*}
 and $c_m$ is a normalizing constant. By the virtue of \eqref{Rel}, one gets 
\begin{multline*}
{}_0F_1\left(\frac{p}{2}; A; -C^TC\right){}_0F_1\left(\frac{p}{2}; B; - C^TC\right) = \\ c_m\int_{\mathbb{R}_+^m} {}_0F_1\left(\frac{p}{2}; D^2; -C^TC\right) \kappa_{A,B}(D)dD 
\end{multline*}
where  
\begin{equation*}
\kappa_{A,B}(D) := V(D) {\bf 1}_{\{\lambda_1 > \cdots > \lambda_m > 0\}}\int_{O(m) \times O(m)} f_{(O_1\sqrt{A}O_1^T, O_2\sqrt{B}O_2^T)}(D)dO \otimes dO.
\end{equation*}
Finally, one performs a change of variable $\lambda_i \mapsto \sqrt{\lambda_i}, 1 \leq i \leq m$ in order to get the product formula: 
\begin{multline*}
{}_0F_1\left(\frac{p}{2}; A; -C^TC\right){}_0F_1\left(\frac{p}{2}; B; - C^TC\right) = \\ \frac{c_m}{2^m}\int_{\lambda_1 > \cdots > \lambda_m > 0} {}_0F_1\left(\frac{p}{2}; D; -C^TC\right) \frac{\kappa_{A,B}(\sqrt{D})}{\sqrt{\lambda_1\dots\lambda_m}}
\prod_{i=1}^md\lambda_i. 
\end{multline*}

\proof[Acknowledgements]
Authors thank P. Graczyk for fruitful discussions held at Angers university and D. St. P. Richards for valuable comments. This work is supported by Agence Nationale de la Recherche grant ANR-09-BLAN-0084-01.

\end{document}